\documentclass[11pt]{amsart}

\usepackage{texmaX,semmaX,semtkzX}
\usepackage{makecell}
\usepackage{multirow}

\usepackage{booktabs} 
\usepackage{mathtools}
\usepackage{amssymb}
\usepackage{listings}

\usepackage{verbatim}  
\usepackage{enumitem}

\usepackage{hyperref}
\newcommand{\res}{\operatorname{res}}
\usepackage{float}

\usepackage{lmodern}
\usepackage{minted} 

\usepackage{listings}
\usepackage[T1]{fontenc}

\usepackage[table]{xcolor}

\usepackage[margin=1in]{geometry}

\usepackage{amsthm}

\DeclareFontFamily{U}{wncy}{}
\DeclareFontShape{U}{wncy}{m}{n}{<->wncyr10}{}
\DeclareSymbolFont{mcy}{U}{wncy}{m}{n}
\DeclareMathSymbol{\Sha}{\mathord}{mcy}{"58}

\DeclareMathOperator{\Ker}{Ker}

\usepackage{tikz}
\usetikzlibrary{arrows.meta}
\usetikzlibrary{arrows.meta}
\usetikzlibrary{positioning, arrows.meta}

\usepackage[utf8]{inputenc}

\usepackage{placeins}

\usepackage{float}
\usepackage{caption}

\usepackage{xcolor}
\usepackage{ulem}

\begin{document}

\title{Nontrivial torsion in the Tate--Shafarevich group of elliptic curves via visibility and twists}
\author{Asuka Shiga}
\address{Mathematical Institute, Graduate School of Science, Tohoku University, 6-3 Aramaki Aza Aoba, Aoba-ku, Sendai, Miyagi 980-8578, Japan.}
\email{otheiio323.com@gmail.com}

\begin{abstract}
Let $\ell$ be an odd prime number. We study the visibility theorem for  certain elliptic curves over $\mathbb{Q}$ with additive reduction at $\ell$, and deduce the existence of nontrivial $\ell$-torsion in $\Sha(E^D/\mathbb{Q})$ for suitable quadratic twists $E^D$. Assuming the Bunyakovsky conjecture for $D_k:=3796416000k^3+6840849600k^2-19745252280k-36525288487$, we prove that there exist infinitely many pairs of BSD twins with identical minimal discriminants and Kodaira symbols, both of which have nontrivial $3$-primary components of their Tate--Shafarevich groups. Via the theory of visibility, this provides an instance of the
infinitude predicted by the Bhargava–Klagsbrun–Lemke Oliver–Shnidman conjecture, within the family
of quadratic twists of an elliptic curve without a rational 3-isogeny.
\end{abstract}

\keywords{visibility, Tate--Shafarevich group, elliptic curves, quadratic twists, Selmer groups, bad reduction}

\maketitle

\setcounter{tocdepth}{2}
\tableofcontents

\section{Introduction}
Let $\ell$ be an odd prime number. Let \(E/\mathbb{Q}\) be a certain elliptic curve without a degree $\ell$-isogeny defined over \(\mathbb{Q}\). We produce square-free integers \(D\) such that \(\Sha(E^D/\mathbb{Q})[\ell]\) is nontrivial, where \(\Sha(E^D/\mathbb{Q})\) denotes the Tate--Shafarevich group of the quadratic twist \(E^D/\mathbb{Q}\) and \(\Sha(E^D/\mathbb{Q})[\ell]\) its \(\ell\)-torsion subgroup. For elliptic curves admitting a degree \(\ell\)-isogeny defined over \(\mathbb{Q}\), the subgroup \(\Sha(E^D/\mathbb{Q})[\ell]\) can be controlled via the Tamagawa ratio. The situation is treated in the literature (\cite{Bhargava}, \cite{Ariel}). However, we focus on the case where there is no isogeny of degree $\ell$ and $E/\mathbb{Q}$ has bad reduction at $\ell$, a situation that cannot be handled using the Tamagawa ratio. It is conjectured that, for an abelian variety $A/\mathbb{Q}$, the group $\Sha(A^D/\mathbb{Q})[\ell]$ is nontrivial for a positive proportion of square-free integers $D$ (the Bhargava–Klagsbrun–Lemke Oliver–Shnidman conjecture; see \cite{Bhargava}). However, to the best of our knowledge, no such positive-density result is known for any elliptic curve $E/\mathbb{Q}$ that admits no $\mathbb{Q}$-rational $\ell$-isogeny.

Our main aim is to establish the nontriviality of \(\Sha(E^D/\mathbb{Q})[\ell]\) for suitable quadratic twists. We obtain a criterion using an auxiliary \(\ell\)-congruent elliptic curve and a relaxed local condition at \(\ell\). Under the rank hypothesis, this criterion also identifies nontrivial classes arising from rational points on the auxiliary curve via visibility.

Cremona and Mazur introduced the notion of visibility. In the examples of \cite{Ma}, the authors observe that certain nontrivial elements of the Tate--Shafarevich group can be explained by the difference between the Mordell--Weil ranks of $E$ and an elliptic curve $F$ with the same mod $3$ Galois representation, namely $E[3] \cong F[3]$ as $G_{\mathbb{Q}}$-modules. Agashe and Stein proved that, for a suitable pair \((E,F)\) of elliptic curves, a positive difference between the Mordell--Weil ranks of \(F\) and \(E\) yields nontrivial elements of \(\Sha(E/\mathbb{Q})\) via visibility, unconditionally, by checking certain local conditions; see \cite{visibility}, Theorem~\ref{sha}. Agashe and Biswas relaxed the hypotheses involving the ambient abelian variety by using the Kummer theory of Néron models \cite{2013}. However, when one is interested in the $\ell$-primary part of the Tate--Shafarevich group for primes $\ell$ dividing the conductor of $E$ or $F$, their theorem does not apply directly. By showing $\operatorname{res}_{\ell} \circ \varphi = 0$ where $\varphi$ is the map defined in Section \ref{phidef} which plays a crucial role in the theory of visibility, Fisher showed that the theorem of Agashe--Stein applies in the situation where $E$ and $F$ have split multiplicative reduction at $\ell$, or where $E$ (resp. $F$) has non-split multiplicative reduction and $F$ (resp. $E$) has good reduction \cite{Fisher}. By contrast, there is no general theory that covers all cases of additive reduction.

We prove the following statement, which remains applicable even in the
case of bad reduction, provided that
\(E(\mathbb{Q}_{\ell})[\ell]=0\). Here, $\mathrm{Vis}_{J^D} (\Sha(E^D/\mathbb{Q}))[\ell]$ is a certain subgroup of $\Sha(E^D/\mathbb{Q}))[\ell]$ which comes from the Mordell--Weil group of $F$, see the definition \ref{vis}.

\begin{theorem}[A visibility theorem for twisted elliptic curves,
including primes of bad reduction]
\label{visiblemain}

Let $E,F$ be elliptic curves over $\mathbb{Q}$.
Fix an odd prime number $\ell$.
Assume that $(E,F)$ satisfies the following conditions:
\begin{enumerate}
  \item
  $E[\ell]=F[\ell]$ as $G_{\mathbb{Q}}$-modules.

  \item
  $\displaystyle
  \gcd\!\left(
    \ell,\;
    \prod_{p<\infty}
    c_{E/\mathbb{Q}_p}\cdot c_{F/\mathbb{Q}_p}
  \right)=1
  $
  and $E(\mathbb{Q}_{\ell})[\ell]=0$.

  \item
  $\displaystyle
  \mathrm{rank}(F/\mathbb{Q})
  -\mathrm{rank}(E/\mathbb{Q})\ge 2
  $
  (resp.\ $\displaystyle
  \mathrm{dim}_{\mathbb{F}_{\ell}}
  \mathrm{Sel}^{\ell}(F/\mathbb{Q})
  -\mathrm{rank}(E/\mathbb{Q})\ge 2
  $).
\end{enumerate}

Then the following statements hold:
\begin{enumerate}
  \item[(a)]
  We have
  \[
    \mathrm{Vis}_{J}\bigl(\Sha(E/\mathbb{Q})\bigr)[\ell]
    \neq 0
    \qquad
    \bigl(\text{resp. }\Sha(E/\mathbb{Q})[\ell]\neq 0\bigr).
  \]

  \item[(b)]
  Let $D\in\mathbb{Z}$ be a square-free integer such that
  $(E^D,F^D)$ satisfies $(2)$ and the corresponding
  inequality in $(3)$.
  Then
  \[
    \mathrm{Vis}_{J^D}\bigl(\Sha(E^D/\mathbb{Q})\bigr)[\ell]
    \neq 0
    \qquad
    \bigl(\text{resp. }\Sha(E^D/\mathbb{Q})[\ell]\neq 0\bigr).
  \]
  
\end{enumerate}

\end{theorem}

The condition \(E^D(\mathbb{Q}_{\ell})[\ell]=0\) may be regarded as mild,
since, by Remark~\ref{localtorsion}, this condition alone can be achieved
by choosing \(D\) appropriately.

The parity of  $\operatorname{rank}(F/\mathbb{Q})-\operatorname{rank}(E/\mathbb{Q})$ is constrained by the hypotheses under which visibility can be applied. In fact, we show in Theorem~\ref{parityeven} that the existing visibility theorems apply only when this rank difference for the \(\ell\)-congruent elliptic curves is even.

Differences in the local conditions at $3$ can indeed lead to non-isomorphic Selmer groups (see Theorem \ref{trueexamplejustify}). Regardless of whether this occurs, relaxing the local condition at $3$ allows us to establish the nontriviality of the $3$-primary part of the Tate–Shafarevich group without analyzing this potentially complicated local condition.

\begin{theorem}\label{trueexamplejustify}
Assume that the Tate--Shafarevich groups under consideration are finite.
Let $E_1 := 38025.\mathrm{ck}1$ and $F := 4225.\mathrm{k}1$
be as in Example~\ref{example3}.
A Weierstrass model for $F^{1190}$ is given by
\[
F^{1190}\colon y^2 = x^3 - x^2 - 315412673x - 2155874804383.
\]
Then, for each $D=1190$, we have
\[
\mathrm{Sel}^3(E_1^D/\mathbb{Q})
\neq \mathrm{Sel}^3(F^D/\mathbb{Q})
\qquad\text{and}\qquad
\Sha(E_1^D/\mathbb{Q})[3] \neq 0.
\]
\end{theorem}

We also produce pairs of non-isomorphic elliptic curves with isomorphic non-trivial Tate–Shafarevich groups and other invariants related to the BSD conjecture in terms of visibility. Let us define
\[
\mathrm{BSD}(E/\mathbb{Q})
:=
\bigl(L(E,s),\,E(\mathbb{Q}),\,\mathrm{Reg}(E/\mathbb{Q}),\,\Omega_E,\,(c_{E/\mathbb{Q}_p})_p,\,\Sha(E/\mathbb{Q})\bigr),
\]
where $L(E,s)$ is the Hasse--Weil $L$-function, $\Omega_E$ is the real period of a global minimal differential,
$(c_{E/\mathbb{Q}_p})_p$ are the Tamagawa numbers. 
We call a pair $(E_1,E_2)$ of non-isomorphic elliptic curves over $\mathbb{Q}$ such that $\mathrm{BSD}(E_1/\mathbb{Q})=\mathrm{BSD}(E_2/\mathbb{Q})$ a BSD twin. In \cite{shiga}, it was proved that there exist infinitely many BSD twins such that $j(E_1)\neq j(E_2)$, $\mathrm{BSD}(E_1/\mathbb{Q})=\mathrm{BSD}(E_2/\mathbb{Q})$, and the Kodaira symbols at every prime and the minimal discriminants are the same. Here, “=” between groups means that the groups are isomorphic as groups. Explicit examples of such pairs are listed in the following table (see \cite{shiga}). Let \(E_1 \coloneqq 38025.ck1\) and \(E_2 \coloneqq 38025.ck2\); they are degree $2$ isogenous over \(\mathbb{Q}\).

\FloatBarrier
{\scriptsize
 \begin{table}[htbp]

\caption{Pairs of elliptic curves sharing the BSD invariants, Kodaira symbols and minimal discriminants. Here, $D>0$ is a prime number with $D\equiv 1 \bmod{8}$, $\bigl(\tfrac{D}{5}\bigr)=-1$, $\bigl(\tfrac{D}{13}\bigr)=1$, and $\gcd(D,2\cdot3\cdot5\cdot13)=1$.}

\centering
\label{tab:elliptic_curves_38025_b}
\begin{tabular}{lll}
    \toprule
    \textbf{Elliptic curve} & $E^D_1: y^2=x^3+25350Dx^2+2471625D^2x$ & $E^D_2: y^2=x^3-50700Dx^2+632736000D^2x$ \\
    \midrule
 j-invariant& $257^3$  & $17^3$  \\  
  
    Mordell--Weil group & 
    $\mathbb{Z}/2\mathbb{Z}$ & $\mathbb{Z}/2\mathbb{Z}$ \\
    Regulator & 1 & 1 \\
    Real period & $\dfrac{0.209\cdots}{\sqrt{D}}$ & $\dfrac{0.209\cdots}{\sqrt{D}}$ \\
    Tamagawa number & $2(p=3),2(p=5), 2(p=13), 2(p=D)$ & $2(p=3), 2(p=5), 2(p=13), 2(p=D)$\\

   Kodaira symbol  & $I_0^*(p=3), III^*(p=5), III^*(p=13)$,  & $I_0^*(p=3), III^*(p=5), III^*(p=13)$ \\

    &  $I_0^* (p=D)$ &  $I_0^* (p=D)$\\        
     \rowcolor{gray!20}  $\Sha(E^D_i/\mathbb{Q})[2^{\infty}]$ & 0 & 0 \\
 \rowcolor{gray!20} $\Sha(E^D_i/\mathbb{Q})$ & $\dagger$ & $\dagger$ \\
    \bottomrule
\end{tabular}

\end{table}}

\FloatBarrier

Both \(E_1\) and \(E_2\) have additive reduction of type \(I_0^*\) at \(3\). For the row marked \(\dagger\), \cite{shiga} shows that
\(\Sha(E_1^D/\mathbb{Q}) \cong \Sha(E_2^D/\mathbb{Q})\)
by choosing \(D\) such that
\(\Sha(E_1^D/\mathbb{Q})[2^{\infty}] = \Sha(E_2^D/\mathbb{Q})[2^{\infty}] = 0\).

Using the Gouvêa--Mazur construction of twists of rank greater than $1$(\cite{GMazur}), we prove the following main theorem of this paper.

\begin{theorem}\label{mainBSDtwins}
Assume the Bunyakovsky conjecture for $D_k:=3796416000k^3+6840849600k^2-19745252280k-36525288487$.
Then there exist infinitely many BSD twins with
$3$-primary part of the Tate--Shafarevich group containing $(\Bbb{Z}/3\Bbb{Z})^2$.
\end{theorem}

For any prime $\ell$ and any abelian variety $A$ over $\Bbb{Q}$, it is predicted by \cite{Bhargava} that there exist positive density square-free $D$ such that $\Sha(A^D/\Bbb{Q})[\ell]\neq 0$. For an odd prime $\ell=3$, in the case where $A$ is an elliptic curve over $\Bbb{Q}$, not a single elliptic curve $E/\Bbb{Q}$ without a $3$-isogeny over $\Bbb{Q}$ is known to satisfy the conjecture. Moreover, not a single elliptic curve $E/\Bbb{Q}$ without a $3$-isogeny over $\Bbb{Q}$ is known to have infinitely many square-free $D\in\Bbb{Z}$ such that $\Sha(E^D/\Bbb{Q})[3]\neq 0$. This is because the Tamagawa ratio does not work as a mechanism in this case. In our twins, we do not have a $3$-isogeny over $\Bbb{Q}$; thus, the main theorem gives an infinitude version of the conjecture via visibility. Note that visibility yields a density-zero family, because we need the rank to be at least $2$.

\subsection{Strategy of the proofs.}
Given the pair of elliptic curves $E_1$ and $E_2$, we choose an auxiliary elliptic curve $F$ such that
\[
E_1[\ell]\cong F[\ell]
\]
as $G_{\mathbb Q}$-modules. In our application, $\ell=3$. Roughly speaking, the Selmer groups $\operatorname{Sel}^{\ell}(E_1/\mathbb Q)$ and $\operatorname{Sel}^{\ell}(F/\mathbb Q)$ are closely related: under suitable local hypotheses, including the assumption that $\ell$ does not divide the product of the Tamagawa numbers of $E_1$ and $F$, their defining local conditions agree away from $\ell$. 

To analyze how ``close'' the $\ell$-torsion subgroup of the Tate--Shafarevich group remains to its original when we ignore the local condition at $\ell$, we study a relaxed version of it in Proposition \ref{V}. Under the mild assumption $E_1(\mathbb{Q}_\ell)[\ell]=0$, we show that the relaxed version is close to the original one. We also prove that the resulting inequality in Proposition \ref{V} is optimal.

For simplicity, we assume finiteness throughout this outline. The actual proof, however, does not require this assumption and is formulated in terms of Selmer groups and the Key Diagram.

To avoid assuming the finiteness of the Tate--Shafarevich group and to keep our presentation consistent with existing proofs of the existence of nontrivial classes via visibility maps, we present our argument using the Key Diagram. In essence, however, the argument is the same as the one outlined above.

Under the additional assumption that $E_1(\mathbb Q_\ell)[\ell]=0$, we can control the discrepancy between the two Selmer groups and obtain a nontrivial class in $\Sha(E_1/\mathbb Q)[\ell]$ whenever
\[
\operatorname{rank}F(\mathbb Q)-\operatorname{rank}E_1(\mathbb Q)\geq 2.
\]
This is the argument developed in Section~4.2, where we relax the local condition at $\ell$.

In the setting of Theorem~\ref{visiblemain}, the rank difference of original curve and auxiliary curve is typically even, as in the example $(E_1,F)=(38025.ck1,38025.i1)$.
The challenge is therefore to produce infinitely many twists for which this difference is at least $2$.

Our approach is to construct a point of infinite order on a polynomial family of quadratic twists and then specialize. For specializations in which this point remains of infinite order, the rank is positive. If we can also show that the root number is $+1$, the applicable parity theorem implies that the rank is even, and hence at least $2$. When the corresponding twist of $E_1$ has rank $0$, this gives the required rank difference. A key step in the proof is to ensure that $\operatorname{rank}(E_1^D/\mathbb{Q})=0$ while simultaneously obtaining $\operatorname{rank}(F^D/\mathbb{Q})\geq 1$ for infinitely many twisting parameters $D$. To achieve this, we assume Bunyakovsky’s conjecture for a suitable polynomial obtained from the right-hand side of a Weierstrass equation for the auxiliary curve $F$. This guarantees that the polynomial takes prime values infinitely often, allowing us to satisfy both rank conditions simultaneously.

\section{Notation}

Let us fix the notation as follows:

\begin{itemize}

\item For an abelian group $A$ and a prime number $p$,
{\setlength{\jot}{1pt}
\[
  \begin{aligned}
    A[p]          &:= \{\,a \in A \mid pa = 0\}
                  &&\text{($p$-torsion subgroup)}\\
    A[p^{\infty}] &:= \{\,a \in A \mid \exists\,n\ge 1,\; p^{n}a = 0\}
                  &&\text{($p$-primary component)}
  \end{aligned}
\]
}

\item Let $\ell$ be a prime and let $E/\mathbb{Q}$ be an elliptic curve. We write
$\mathrm{Sel}^{\ell}(E/\mathbb{Q})$ for the $\ell$-Selmer group of $E$ over
$\mathbb{Q}$.

\item $\Sha(E/\mathbb{Q})$: the Tate--Shafarevich group of $E/\mathbb{Q}$.

\item $\#\Sha(E/\mathbb{Q})_{\mathrm{an}}$: the analytic order of the Tate--Shafarevich group of $E/\mathbb{Q}$, i.e.\ the order predicted by the BSD conjecture. Note that the LMFDB database lists the analytic order of the Tate–Shafarevich group, not the Tate–Shafarevich group itself.

\item Let $A/\Bbb{Q}$ be an abelian variety. We define the set of $\bar{\Bbb{Q}}/\Bbb{Q}$-twists of $A$ by
\[
\text{Twist}(A/\Bbb{Q})
:=
\frac{\{\,B/\Bbb{Q}\mid B \text{ is } \text{isomorphic to } A \ \text{over} \ \bar{\Bbb{Q}}\,\}}
{\Bbb{Q}\text{-isomorphism}}.
\]

Let $D$ be a square-free integer. We define the quadratic twist $A^D/\Bbb{Q}$ to be the image of the class of $D$
under the composite map
\[
\Bbb{Q}^{\times}/(\Bbb{Q}^{\times})^{2}
\;\xrightarrow{\ \sim\ }\;
H^{1}\!\left(\Gal(\bar{\Bbb{Q}}/\Bbb{Q}),\mu_{2}\right)
\longrightarrow
H^{1}\!\left(\Gal(\bar{\Bbb{Q}}/\Bbb{Q}),\Aut_{\bar{\Bbb{Q}}}(A)\right)
\;\xrightarrow{\ \sim\ }\;
\text{Twist}(A/\Bbb{Q}).
\]

The following is a special case and an explicit parameterization of the above definition of $A^D$.

\item Let \(E/\mathbb{Q}\) be an elliptic curve given by a minimal Weierstrass equation
\[
  y^2 + a_1xy + a_3y = x^3 + a_2x^2 + a_4x + a_6.
\]
Let $D$ be a square-free integer. Then a Weierstrass equation for $E^D$ is: 

\begin{align*}
E^D: y^2 &+ a_1xy +a_3y= \\ &= x^3 + \left(a_2 D + a_1^2 \frac{D-1}{4}\right) x^2 + \left(a_4 D^2 + a_1 a_3 \frac{D^2-1}{2} \right)x + \left(a_6 D^3 + a_3^2 \frac{D^3-1}{4}\right).
\end{align*}

\item For an elliptic curve $E/\Bbb{Q}$, $N_E$ denotes the conductor of $E$.

\item 
For a field $K$,  $\mathrm{Gal}(\overline{K}/K)$-module $M$, we abbreviate $H^1(\mathrm{Gal}(\overline{K}/K), M)$ to $H^1(K, M)$.

\item For an elliptic curve $E/\mathbb{Q}$, set
\[
E_0(\mathbb{Q}_p) \coloneqq 
\{\, P \in E(\mathbb{Q}_p) \mid \tilde{P} \text{ is non-singular} \} \ \text{where} \ \tilde{P} \ \text{denotes the reduction of}\  P\ \text{modulo} \ p. 
\]

$\mathcal{E}$: the Néron model of $E/\mathbb{Q}_p$.\par
$\mathcal{E}_0$: the identity component of $\mathcal{E}$.\par
$\tilde{\mathcal{E}}, \tilde{\mathcal{E}}_0$: the special fibers of $\mathcal{E}$ and $\mathcal{E}_0$ at $p$, respectively.\par

We define $\Phi_E$ to be the finite group scheme $\tilde{\mathcal{E}}/\tilde{\mathcal{E}}_0$.

Then
\[
c_{E/\mathbb{Q}_p}
= [E(\mathbb{Q}_p):E_0(\mathbb{Q}_p)]
= \#(\tilde{\mathcal{E}} / \tilde{\mathcal{E}}_0)(\mathbb{F}_p).
 \]

\item
Let $E/\mathbb{Q}$ be an elliptic curve and let $\tilde E$ denote its reduction modulo $p$.
Define
\[
\tilde{E}_{\mathrm{ns}}(\mathbb{F}_p)
\coloneqq
\{\,P\in \tilde{E}(\mathbb{F}_p)\mid P \text{ is nonsingular}\,\}.
\]
Let $E_1(\mathbb{Q}_p)$ denote the kernel of the reduction map
$E(\mathbb{Q}_p)\to \tilde{E}_{\text{ns}}(\mathbb{F}_p)$.

\item 
Fix a prime number $\ell$.
We say that abelian varieties $A$ and $B$ over $\mathbb{Q}$ are $\ell$-congruent
if
\[
A[\ell]=B[\ell]
\]
as $G_{\mathbb{Q}}$-modules.

\end{itemize}

\section{$\ell$-congruent elliptic curves}

In this section, for an elliptic curve $E/\mathbb{Q}$, we recall the Rubin--Silverberg family $F_t$ $(t\in\mathbb{Q})$ such that $E[3]\cong F_t[3]$ as $G_{\mathbb{Q}}$-modules.

\begin{theorem}[Rubin--Silverberg, \cite{RS95}]\label{RS}
Let $E \colon y^2=x^3+ax+b$ be an elliptic curve defined over $\Bbb{Q}$.

Set
\[
J_E:=\frac{j(E)}{1728}
 ,
\lambda:=J_E-1
 .
\]

For each $t\in\mathbb{Q}$, define
\[
F_t\colon y^2=x^3+A(t)x+B(t),
\]
where
\[
A(t)
=
a\left(
-3\lambda^2t^4
-8\lambda^2t^3
+6\lambda t^2
+1
\right)
\]
and
\[
\begin{aligned}
B(t)
=
b\bigl(
&-\lambda^3(8J_E+1)t^6
+6\lambda^2(2J_E+1)t^5
+15\lambda^2t^4 \\
&+20\lambda^2t^3
-15\lambda t^2
+6t
+1
\bigr).
\end{aligned}
\]

Then $E[3]\cong F_t[3]$. 

\end{theorem}

\begin{lemma}\label{noisogeny}
Let $E_1 \coloneqq 38025.\mathrm{ck}1$.
Since $E_1$ has supersingular reduction at $7$,
we have $a_7(E_1)=0$. Thus, the characteristic
polynomial of Frobenius at $7$ on $E_1[3]$ is
\[
T^2-a_7(E_1)T+7 \equiv T^2+1 \bmod{3},
\]
which is irreducible over $\mathbb F_3$.
Hence $E_1[3]$ is an irreducible
$G_{\mathbb Q}$-module, so $E_1$ admits no
$3$-isogeny defined over $\mathbb Q$.
Since quadratic twisting preserves the
irreducibility of the mod-$3$ representation,
the same holds for $E_1^D$ for every
square-free integer $D$.
\end{lemma}

\begin{example}\label{example2}
\begin{enumerate}
\item Let
\[
t=-\frac{216}{94535}.
\]
Then
\[
F_t=\texttt{38025.i1},
\]
and $D=6977$ satisfies the above conditions.

We can also check this by Sturm bound. 
Let \(f_{\mathrm{ck1}}\) and \(f_{\mathrm{i1}}\) be the modular forms attached to the elliptic curves \[E_1\coloneqq38025.\mathrm{ck}1 \  \text{and} \  F\coloneqq38025.\mathrm{i}1,\] respectively. Their \(q\)-expansions (up to \(O(q^{20})\)) are
\[
\begin{aligned}
f_{\mathrm{ck1}}(q) &= q + q^{2} - q^{4} - 3q^{8} - 2q^{11} - q^{16} - 6q^{19} + O\!\left(q^{20}\right),\\
f_{\mathrm{i1}}(q)  &= q - 2q^{2} + 2q^{4} + 3q^{7} - 5q^{11} - 6q^{14} - 4q^{16} + 3q^{17} + 6q^{19} + O\!\left(q^{20}\right).
\end{aligned}
\]
Moreover, these forms are congruent modulo \(3\); that is,
\[
f_{\mathrm{ck1}} \equiv f_{\mathrm{i1}} \bmod{3}.
\]

The Sturm bound for weight \(2\) and level \(N=38025\) is  
\[
\left\lfloor \frac{2\,[SL_2(\mathbb{Z}):\Gamma_0(N)]}{12}\right\rfloor
=\left\lfloor \frac{1}{6}\,N\prod_{p\mid N}\left(1+\frac{1}{p}\right)\right\rfloor
=10920,
\]
, we verified that the Fourier coefficients of \(f_{\mathrm{ck1}}\) and \(f_{\mathrm{i1}}\) are congruent modulo \(3\) up to this bound by SageMath. Thus, $E_1[3]$ and $F[3]$ have isomorphic semisimplifications. By Lemma \ref{noisogeny}, $E_1[3]\cong F[3]$.

\item Let
\[
t=\frac{1728}{16972865}
  =\frac{1}{\lambda},
\qquad
\lambda:=J-1.
\]
Then $D=17$ satisfies the above conditions.

The Kodaira types at the relevant primes are given by
\[
\begin{array}{c|cccc}
p & 3 & 5 & 13 & 257\\
\hline
\text{Kodaira type}
  & I_3^{*}
  & \mathrm{III}
  & \mathrm{III}
  & I_6\ \text{(non-split)}
\end{array}
\]
\end{enumerate}

\end{example}

\begin{example}\label{example3}
Let
\[
t_0=-\frac{144}{69277}.
\]

This curve has LMFDB label
\[
\texttt{4225.k1}.
\]

\end{example}

\section{Visibility}

\subsection{Theory of visibility}\label{phidef}

We recall the theory of visibility. Let $E$ and $F$ be elliptic curves that are abelian subvarieties of an abelian variety $J$ over $\Bbb{Q}$. There are two approaches to defining the visibility map: the original construction in \cite{visibility}, and the map that is essentially used later in \cite{2013} and \cite{Fisher}. 

\begin{definition} \label{vis}

We define 

\[\text{Vis}_J(H^1(\Bbb{Q},E))\coloneqq \mathrm{Ker}(H^1(\Bbb{Q},E)\to H^1(\Bbb{Q},J))\]

\[\text{Vis}_J(\Sha(E/\Bbb{Q}))\coloneqq \Sha(E/\Bbb{Q}) \cap \text{Vis}_J(H^1(\Bbb{Q},E))\]

\end{definition}

Let $C\coloneqq J/E$ be the quotient abelian variety. The short exact sequence \(0\to E\to J\to C\to 0\) induces
\[
0 \longrightarrow J(\Bbb Q)/E(\Bbb Q) \longrightarrow C(\Bbb Q) \longrightarrow \operatorname{Vis}_J \ \!\bigl(H^1(\Bbb Q,E)\bigr)\to 0.
\]

Suppose that \(i_F(F[n])=i_E(E[n])\) inside \(J\) as \(G_{\Bbb Q}\)-submodules (where \(i_E:E\hookrightarrow J\) and \(i_F:F\hookrightarrow J\) are the inclusions), and let \(\theta:F[n]\xrightarrow{\sim}E[n]\) be the induced isomorphism.
Then we have a commutative diagram below. Note that \((a)\) is commutative since \(i_E|_{E[n]}\circ \theta=i_F|_{F[n]}\).

\[
\begin{tikzcd}[column sep=small, row sep=large]
0 \arrow[r] & F[n] \arrow[r,hook] \arrow[d,hookrightarrow] \arrow[dr, phantom, "(a)"] & F \arrow[r,"{[n]}"] \arrow[d,hookrightarrow] &
F \arrow[r] \arrow[d] & 0 \\
0 \arrow[r] & E \arrow[r,hook,"i_E"] & J \arrow[r] & C \arrow[r] & 0
\end{tikzcd}
.\]
Hence we have the following commutative diagram with exact rows under the assumption that \(F(\Bbb{Q})[n]\) is trivial:
\[
\begin{tikzcd}[column sep=small, row sep=large]
0 \arrow[r] & F(\Bbb{Q})\arrow[r,"{f\coloneqq[n]}"] \arrow[d] & F(\Bbb{Q}) \arrow[r] \arrow[d] &
F(\Bbb{Q})/nF(\Bbb{Q}) \arrow[r] \arrow[d,"\varphi"] & 0 \\  
0 \arrow[r] & J(\Bbb{Q})/E(\Bbb{Q}) \arrow[r,hook,"g"] & C(\Bbb{Q}) \arrow[r] & \operatorname{Vis}_J\ \!\bigl(H^1(\Bbb{Q},E)\bigr) \arrow[r] & 0
\end{tikzcd}.
\]

Here, $\varphi$ denotes the unique map induced by $\mathrm{Coker}\  f$ to $\mathrm{Coker}\ g$. For this \(\varphi\), Theorem~3.1 of \cite{visibility} establishes the following.

\begin{theorem}[Lemma 3.7 of \cite{visibility} for elliptic curves]\label{H^1}

Let $E$ and $F$ be elliptic curves that are abelian subvarieties of an abelian variety $J$ over $\Bbb{Q}$.
Assume there exists $(E,F,J)$ such that
\begin{enumerate}
\item $E\cap F$ is finite. 
\item  $n$ is an odd integer,
and that
\[
\gcd\!\Bigl(n,\;  \#(J/F)(\Bbb{Q})_{\mathrm{tor}}\cdot \#F(\Bbb{Q})_{\mathrm{tor}}\Bigr)=1,
\]
.
\item $F[n]=E[n]$ as subgroup schemes of $J$. 

If $E/\Bbb{Q}$ has rank $r$, then there is a natural homomorphism
\[
\varphi:\; F(\Bbb{Q})/nF(\Bbb{Q})\to \mathrm{Vis}_J\ \!\bigl(H^1(\Bbb{Q},E)\bigr)
\] such that $\#\mathrm{Ker}\varphi \le n^r$.
\end{enumerate}

\end{theorem}

\begin{remark}\label{commutative}

Roughly speaking, in our setting, $\mathrm{Sel}^{\ell}(E/\mathbb{Q})$ and $\mathrm{Sel}^{\ell}(F/\mathbb{Q})$ differ only in their local conditions at $\ell$. Since $\mathrm{Sel}^{\ell}(F/\mathbb{Q})$ contains $F(\mathbb{Q})/\ell F(\mathbb{Q})$, which has $\mathbb{F}_{\ell}$-dimension at least two when $\operatorname{rank}F(\mathbb{Q})\geq 2$, a nonzero element survives in $\mathrm{Sel}^{\ell}(E/\mathbb{Q})$. Although this principle is simple, the resulting nontriviality is essentially captured by the composite map $h\circ\theta\circ\kappa$ in (1) of the key diagram below, and is realized geometrically via visibility in the abelian variety $J$ in (2).

We used the ambient $J$ to define the map $\varphi: F(\Bbb{Q})/nF(\Bbb{Q})\to \mathrm{Vis}_J \bigl(H^1(\Bbb{Q},E)\bigr)$, but there is a slightly different guise to define $\varphi$ without $J$. We claim that the map $\varphi$ in Theorem \ref{H^1} is the composition of the Kummer map with the non-canonical isomorphism $\theta: H^1(\mathbb{Q}, F[n])\cong H^1(\mathbb{Q},E[n])$, followed by the surjection $H^1(\mathbb{Q},E[n])\twoheadrightarrow H^1(\mathbb{Q},E)[n]$. Consider the following commutative diagram with exact rows.


\FloatBarrier

\begin{figure}[htbp]
\centering
\[
\begin{tikzcd}[
    column sep={1.8em,1.2em,1.2em},
    row sep={2.2em,2.4em,2.4em},
    every matrix/.append style={name=mycd},
    execute at end picture={
      \path ([xshift=2em]mycd-2-2) --
        node[scale=2] {$\circlearrowright$}
        node[scale=.5] {$(1)$}
        ([xshift=-2em]mycd-3-4);
      \path ([xshift=2em]mycd-3-2) --
        node[scale=2] {$\circlearrowleft$}
        node[scale=.5] {$(2)$}
        ([xshift=-2em]mycd-4-4);
    }
]
& &
\operatorname{Sel}^{n}(F/\mathbb{Q})
  \ar[dr,hook]
  \ar[dddr,bend left=65,"\Psi"]
& \\
\operatorname{coker} f
  \ar[r,"\sim"]
&
\dfrac{F(\mathbb{Q})}{nF(\mathbb{Q})}
  \ar[ur,hook]
  \ar[rr,"\kappa"']
  \ar[dd,"\varphi"']
  \ar[ddrr]
& &
H^{1}(\mathbb{Q},F[n])
  \ar[d,"\theta"]
\\
& {} & &
H^{1}(\mathbb{Q},E[n])
  \ar[d,two heads,"h"]
\\
(\operatorname{coker} g)[n]
  \ar[r,"\sim"]
&
\bigl(\mathrm{Vis}_{J}H^{1}(\mathbb{Q},E)\bigr)[n]
  \ar[rr,hook]
& &
H^{1}(\mathbb{Q},E)[n].
\end{tikzcd}
\]
\caption{Key diagram: Two ways to obtain visible elements in $\Sha(E/\Bbb{Q})[n]$.}
\label{twoway}
\end{figure}

\FloatBarrier

Coker

In particular, for all $x \in F(\mathbb{Q})/nF(\mathbb{Q})$, the relation $(h \circ \theta \circ \kappa)(x) = \varphi(x)$ holds. 

Viewing $\varphi$ in this way, since $\Ker(\varphi)
=\kappa^{-1}\!\bigl(\theta^{-1}(\Ker(h))\bigr)$, it is easy to see that $\#\mathrm{Ker}\varphi \le \#(E(\mathbb{Q})/nE(\mathbb{Q})) \le \#E(\mathbb{Q})[n] \cdot n^{\mathrm{rank}(E/\Bbb{Q})}$. For the more treatment of the composition $h\circ \theta \circ \kappa$, see \cite{2013} and \cite{Fisher}.\end{remark}

\begin{theorem}[Theorem 3.1 of \cite{visibility}, visibility Theorem of Agashe--Stein]\label{sha}
Let \(N_J\) be an integer divisible by the residue characteristics of all primes of bad reduction for \(J\).
In addition to conditions \((1)\)–\((3)\) of Theorem \ref{H^1}, assume that
\[
\gcd \Bigl(
  n,\;
  N_J\cdot \prod_{p}\!\bigl(c_{E/\Bbb{Q}_p}\ \cdot c_{F/\Bbb{Q}_p}\bigr)\Bigr)=1.
\]
If \(E/\mathbb{Q}\) has rank \(r\), then there is a natural homomorphism
\[
\varphi:\; F(\mathbb{Q})/nF(\mathbb{Q}) \to \mathrm{Vis}_J\! \ \bigl(\Sha(E/\Bbb{Q})\bigr)
\]such that $\#\mathrm{Ker}\varphi \le n^r$.
\end{theorem}

\maketitle

If one could produce a difference of ranks at least \(1\), then this visibility theorem would yield a nontrivial element in \(\Sha(E/\mathbb{Q})[\ell]\). In fact, however, the following result shows that the difference must be at least \(2\). This implies that, within the framework of the existing visibility theory, one must obtain an even difference of ranks of at least $2$.

The following are the formal consequence from the results of Agashe--Stein and Fisher establishing the triviality of $\mathrm{res}_{\ell}\circ \phi$. For an alternative proof in case (1), see Remark~1.10 of \cite{flat}.

\begin{theorem}\label{parityeven}
Let $\ell$ be an odd prime number, and let $E/\mathbb{Q}$ and $F/\mathbb{Q}$ be elliptic curves such that
\[
E[\ell]\simeq F[\ell]
\]
as $G_{\mathbb{Q}}$-modules.
Assume that $\ell\nmid c_{E/\mathbb{Q}_p}$ and $\ell\nmid c_{F/\mathbb{Q}_p}$ for every prime $p$, and that one of the following conditions holds:
\begin{enumerate}
    \item both $E$ and $F$ have good reduction at $\ell$;
    \item both $E$ and $F$ have split multiplicative reduction at $\ell$;
    \item one of $E$ and $F$ has non-split multiplicative reduction at $\ell$, and the other has good reduction at $\ell$.
\end{enumerate}

If moreover $\Sha(E/\mathbb{Q})[\ell^\infty]$ and $\Sha(F/\mathbb{Q})[\ell^\infty]$ are finite, then
\[
\operatorname{rank}(E/\mathbb{Q})\equiv \operatorname{rank}(F/\mathbb{Q}) \bmod{2}.
\]
\end{theorem}

\begin{proof}

It suffices to prove that the local conditions agree at every prime
$p$, as well as at the archimedean place. When $p\neq \ell$, the
assumption on the Tamagawa numbers implies that the local images of
the Kummer maps at $p$ are the unramified subgroups
\[
H^1_{\mathrm{nr}}(\mathbb{Q}_p,E[\ell])
=
H^1_{\mathrm{nr}}(\mathbb{Q}_p,F[\ell]),
\]
and hence they agree.

At the archimedean place, since $\ell$ is odd, we have
\[
E(\mathbb{R})/\ell E(\mathbb{R})
=
F(\mathbb{R})/\ell F(\mathbb{R})
=
0,
\]
so the local conditions agree there as well.

It remains to consider the prime $\ell$. If both $E$ and $F$ have good reduction at $\ell$, then $\delta_\ell=0$ by \cite{visibility}. It is enough to prove that $\mathrm{Im}(\kappa_{E,{\ell}})\subset \mathrm{Im}(\kappa_{F,\ell})$. Consider the following commutative diagram, whose commutativity follows from Remark \ref{commutative} after replacing $\Bbb{Q}$ with $\Bbb{Q}_{\ell}$.

{\scriptsize
\[
\begin{tikzcd}[column sep={1.8em,3.4em}, row sep=1.5em]
\dfrac{E(\mathbb{Q}_{\ell})}{\ell E(\mathbb{Q}_{\ell})}
  \ar[r,"{\kappa_{E,\ell}}"]
  \ar[dd]
  \ar[ddr,"\varphi_{\ell}\coloneqq h\circ \theta \circ {\kappa_{E,\ell}}"'] &
H^{1}(\mathbb{Q}_{\ell},E[\ell]) \ar[d,"\theta: \cong"] \\
{} & H^{1}(\mathbb{Q}_{\ell},F[\ell]) \ar[d, two heads,"h"] \\
(\mathrm{Vis}_{J} H^{1}(\mathbb{Q}_{\ell},F))[\ell] \ar[r,hook] &
H^{1}(\mathbb{Q}_{\ell},F)[\ell].
\end{tikzcd}
\]}

Here, $J \stackrel{\text{def}}{=} (F \times E)/\Delta, \Delta =F[\ell]=E[\ell]$ and \[\text{Vis}_J(H^1(\Bbb{Q}_{\ell},F)) \stackrel{\text{def}}{=}  \mathrm{Ker}(H^1(\Bbb{Q}_{\ell},F)\to H^1(\Bbb{Q}_{\ell},J)).\]

Case 3 of the proof of Theorem 3.1 of \cite{visibility} proves $\varphi_{\ell}=0$.

If both $E$ and $F$ have split multiplicative reduction at $\ell$, then
$\delta_\ell=0$ by Theorem 4.2 of \cite{Fisher}. On the other hand, if one
of $E$ and $F$ has non-split multiplicative reduction at $\ell$ and the
other has good reduction at $\ell$, then $\delta_\ell=0$ by Theorem 4.4 of
\cite{Fisher}.

Therefore, the local conditions are identical at every
place, and hence the Selmer groups coincide. Assuming that \(\#\Sha(E/\mathbb{Q})[\ell^\infty]\) and \(\#\Sha(F/\mathbb{Q})[\ell^\infty]\) are finite, we have \(\operatorname{rank}(E/\mathbb{Q})\equiv \operatorname{rank}(F/\mathbb{Q}) \bmod{2}\).
\end{proof}

\begin{lemma}\label{notthinghamparity}
Let $F:=38025.i1$. Suppose that $D>0, D\equiv 1 \bmod 4, \mathrm{gcd}(D, N_{F})=1$. 
Then
\[
 \mathrm{dim}_{\Bbb{F}_{\ell}} {\mathrm{Sel}}^{\ell}(F^D/\Bbb{Q})\equiv 0 \bmod{2}.
\]
\end{lemma}

\begin{proof}
Let $\omega(F)$ denote the root number of $F/\Bbb{Q}$.
By the LMFDB, $\omega(F)=1$.
We have
\[
\omega(F^D)
=\left(\frac{-N_F}{D}\right)\omega(F).
\]
Since
\[
\left(\frac{-N_{F}}{D}\right)
=\left(\frac{-1}{D}\right)\left(\frac{N_{F}}{D}\right)
=\left(\frac{195^2}{D}\right)=1,
\]
we obtain $\omega(F^D)=\omega(F)$.
\end{proof}

\subsection{The relaxed version of $\ell$-torsion subgroup of the Tate--Shafarevich group}

As we have seen in the previous section, for a prime $\ell$ of bad reduction, $\text{res}_{\ell} \circ \varphi$ need not be zero. If there exists an isogeny $\phi: E \to E'$ over $\mathbb{Q}_{\ell}$ such that its dual $\hat{\phi}$ satisfies $E(\mathbb{Q}_{\ell})/\hat{\phi} E'(\mathbb{Q}_{\ell}) = 0$, then there are cases where the visibility theorem can be applicable since $\text{res}_{\ell} \circ \varphi = 0$ holds \cite{Fisher}. 
By contrast, there is no general theory that covers all cases of additive reduction. For example, even if $E$ has additive reduction (which is not potentially multiplicative), a finite extension $K/\mathbb{Q}_\ell$ can provide good reduction. However, for $\ell=3$, the condition $e_{K/\mathbb{Q}_\ell} < \ell-1$ is never satisfied, preventing a direct application of the visibility theorem. 

In light of the ``inflated Selmer group'' introduced in the appendix to \cite{mazur}, we investigate the $\ell$-torsion subgroup of the relaxed Tate--Shafarevich group.

 \begin{definition}
Let $\ell$ be a prime and let $E/\mathbb{Q}$ be an abelian variety.
Fix a subgroup $M\subset H^1(\mathbb{Q}_\ell,E)[\ell]$.
We define
\[
\Sha(E/\mathbb{Q},M)[\ell]
:=\{\, a\in H^1(\mathbb{Q},E)[\ell]\mid
\text{res}_p(a)=0\ \text{for all }p\le \infty,\ p\neq \ell,\ \text{and }\text{res}_\ell(a)\in M \,\}.
\]

In particular, for $\varphi$ in Remark \ref{commutative}, we denote $\Sha(E/\mathbb{Q},\mathrm{Im}(\text{res}_{\ell}\circ \varphi))[\ell]$ as $\Sha(E/\mathbb{Q},\operatorname{Vis})[\ell]$.

For $M=H^1(\mathbb{Q}_{\ell},E)[\ell]$, we write $R(E/\mathbb{Q})
\stackrel{\mathrm{def}}{=}
\Sha\bigl(E/\mathbb{Q},H^1(\mathbb{Q}_{\ell},E)[\ell]\bigr)[\ell]$.
\end{definition}

\begin{remark}
For the estimates needed in this paper, it would suffice to consider $R(E/\mathbb{Q})$. However, considering only $R(E/\mathbb{Q})$ would leave open the possibility that a stronger estimate could be obtained for classes arising from visibility by relaxing the local conditions. To rule out this possibility, we also consider $\Sha(E/\mathbb{Q},\mathrm{Vis})[\ell]$. 
\end{remark}

\begin{theorem}\label{rewrite}
Let $\ell$ be a prime number. In addition to conditions \((1)\)–\((3)\) of Theorem \ref{H^1}, assume that
\[
\gcd (
  \ell,\;
   \prod_{p}\!\bigl(c_{E/\Bbb{Q}_p}\ \cdot c_{F/\Bbb{Q}_p}\bigr)\Bigr)=1.
\]

Then, we have $\mathrm{Im}\varphi \subset \Sha(E/\mathbb{Q},\operatorname{Vis})[\ell]$.
    
\end{theorem}

\begin{proof}
We prove that $\mathrm{res}_{\ell}\circ \varphi=0$ for the infinite prime $\infty$ and all finite primes $\ell \neq p$. For the proofs regarding $\infty$ and $\ell \neq p$, see Case 1 and Case 2 of \cite{visibility}, respectively.

\end{proof}

The situation is as follows. 

\begin{figure}[H]
\centering
\[
\begin{tikzcd}[row sep=1.2em, column sep=3.2em, cells={font=\small}]
F^{D}(\Q)/\ell F^{D}(\Q) \arrow[dr, hook, "\varphi"'] &
\Sha(E^{D}/\Q)[\ell] \arrow[d, hook] \\
& \Sha(E^D/\mathbb{Q},\operatorname{Vis})[\ell]
\end{tikzcd}
\]
\caption{If $\#(F^D(\Bbb{Q})/\ell F^D(\Bbb{Q}))\ge \ell^2$ and $\#\Sha(E^{D}/\Q, \text{Vis})[\ell]/\#\Sha(E^{D}/\Q)[\ell] < \ell^{2}$, then $\Sha(E^{D}/\Q)[\ell]\neq 0$ when $\mathrm{Ker}\ \varphi=0$.}
\end{figure}

Let us consider the following commutative diagram. 
The upper row is exact by Appendix~2 of \cite{Ca} or Theorem~3.3 of \cite{behavior}, where $\mathrm{Sel}^{\ell}(E/\Bbb{Q})^*$ is the Pontryagin dual of $\mathrm{Sel}^{\ell}(E/\Bbb{Q})$.

\begin{figure}[H]
\centering
{\scriptsize
\begin{tikzcd}[row sep=3em, column sep=2em, /tikz/font=\normalsize]
  & \Sha(E/\Bbb{Q})[\ell] \arrow[r] \arrow[d,"F"]
    & H^1(\Bbb{Q},E)[\ell]
      \arrow[r,"\bigoplus \mathrm{res}_p"]
      \arrow[d,equal]
    & \displaystyle\bigoplus_p H^1(\Bbb{Q}_p,E)[\ell]
      \arrow[d,"\mathrm{proj}"]
      \arrow[r,"\pi"]
    &[-8em]
      \!\mathrm{Sel}^{\ell}(E/\Bbb{Q})^* \arrow[r]
    &[-0.5em] 0
  \\
  0 \arrow[r]
    & \Sha(E/\Bbb{Q},M)[\ell] \arrow[r]
    & H^1(\Bbb{Q},E)[\ell] \arrow[r]
    & \displaystyle
      \left(
        \bigoplus_{\substack{p\leq\infty\\ p\neq\ell}}
        H^1(\Bbb{Q}_p,E)[\ell]
      \right)
      \oplus
      \left(
        H^1(\Bbb{Q}_\ell,E)[\ell]/M
      \right)
    & &
\end{tikzcd}
}

\caption{The commutative diagram used in the proof of Proposition~\ref{V}}
\label{fig:snake}

\end{figure}

Let
\[
P:\bigoplus_p H^1(\mathbb Q_p,E)[\ell]
\longrightarrow
\left(\bigoplus_{p\neq \ell}H^1(\mathbb Q_p,E)[\ell]\right)
\oplus
\left(H^1(\mathbb Q_\ell,E)[\ell]/M\right)
\]
be the homomorphism that is the identity on the summands for
\(p\neq\ell\) and the quotient map on the \(\ell\)-summand. We define
\[
H:=
P\big|_{\operatorname{Im}(\oplus_p\operatorname{res}_p)}
:
\operatorname{Im}\!\left(
H^1(\mathbb Q,E)[\ell]
\xrightarrow{\oplus_p\operatorname{res}_p}
\bigoplus_pH^1(\mathbb Q_p,E)[\ell]
\right)
\longrightarrow
P\!\left(
\operatorname{Im}(\oplus_p\operatorname{res}_p)
\right).
\]

We compare the size of the $\ell$-torsion subgroup of the usual Tate--Shafarevich group with that of the Tate--Shafarevich group in which the local condition at $3$ is relaxed. Although the inequality used in Proposition~\ref{V} is itself elementary, it is natural to ask whether it can be improved. Proposition~\ref{V} shows that, in the case relevant to this paper, namely when $M$ arises from visibility, there is in fact no room for improvement.

\begin{proposition}\label{V}
Let $\ell$ be a prime, let $E/\mathbb{Q}$ be an elliptic curve, and let
$M\subset H^{1}(\mathbb{Q}_{\ell},E)[\ell]$ be a subgroup. Then there is an exact sequence
\[
  0\to \frac{\Sha(E/\mathbb{Q},M)[\ell]}{\Sha(E/\mathbb{Q})[\ell]} \to M\stackrel{\pi \circ i_{\ell}}{\to} \mathrm{Sel}^{\ell}(E/\Bbb{Q})^*.
\] where
\[
i_\ell:\ H^1(\Bbb Q_\ell,E)[\ell]\hookrightarrow \bigoplus_{p} H^1(\Bbb Q_p,E)[\ell]
\]
is the natural injection.

In particular, we obtain
\[
  \frac{\#\Sha(E/\mathbb{Q},M)[\ell]}{\#\Sha(E/\mathbb{Q})[\ell]}
  \le \ell\,\#E(\mathbb{Q}_{\ell})[\ell].
\]

Moreover, we have 

\[
  \frac{\Sha(E/\mathbb{Q},\text{Vis})[\ell]}{\Sha(E/\mathbb{Q})[\ell]}
  \cong  \mathrm{Im}(\mathrm{res}_{\ell}\circ \phi).
\]
\end{proposition}

\begin{proof}

By the diagram Figure~\ref{fig:snake}, we have
\[
\operatorname{Coker} F
 \cong \mathrm{Ker}H.
\]

 Then we have

\[
\begin{aligned}
\Ker H
&=
\operatorname{Im}\!\left(\bigoplus_p \mathrm{res}_p\right)
\cap i_\ell(M) \\
&=
\left\{\,i_\ell(x)\;\middle|\;
x\in M,\ \pi\bigl(i_\ell(x)\bigr)=0
\,\right\}.
\end{aligned}
\]

Thus,  $0\to \frac{\Sha(E/\mathbb{Q},M)[\ell]}{\Sha(E/\mathbb{Q})[\ell]} \to M\to \mathrm{Sel}^{\ell}(E/\Bbb{Q})^*$ is an exact sequence.

Note that $\bigoplus_{p}\mathrm{res}_p\bigl(H^1(\Q,E)[\ell]\bigr)$ is a subgroup of
$\bigoplus_{p} H^1(\Q_p,E)[\ell]$ and hence \[\#\Ker(H)\le \#M.\]

We finally prove that \(\#M \le \ell\,\#E(\Bbb Q_{\ell})[\ell]\). By local Tate duality we have
 $\#H^1(\mathbb{Q}_\ell,E)[\ell]
 = \#\bigl(E(\mathbb{Q}_\ell)/\ell E(\mathbb{Q}_\ell)\bigr)$. By Mattuck's theorem (\cite{Mattuck}), $E(\mathbb{Q}_\ell)\cong \mathbb{Z}_\ell \times E(\mathbb{Q}_\ell)_{\mathrm{tor}}$,
and hence $E(\mathbb{Q}_\ell)/\ell E(\mathbb{Q}_\ell)
 \cong \mathbb{Z}_\ell/\ell\mathbb{Z}_\ell \times E(\mathbb{Q}_\ell)[\ell]$,
so that $\#\bigl(E(\mathbb{Q}_\ell)/\ell E(\mathbb{Q}_\ell)\bigr)
 = \#\bigl(\mathbb{Z}_\ell/\ell\mathbb{Z}_\ell\bigr)\cdot\#E(\mathbb{Q}_\ell)[\ell]
 = \ell\,\#E(\mathbb{Q}_\ell)[\ell]$. Therefore $\#M\le \ell\,\#E(\mathbb{Q}_\ell)[\ell]$.

 Moreover, by Theorem~\ref{rewrite}, we have
\[
  \res_p(\phi(y)) = 0 \qquad \text{for all } p \neq \ell .
\]
Hence, if $x = \res_\ell(\phi(y))$, then
\[
  i_\ell(x)
  = \bigl( \res_p(\phi(y)) \bigr)_p
  = \Bigl( \bigoplus_p \res_p \Bigr)(\phi(y))
  \in \mathrm{Image}\Bigl( \bigoplus_p \res_p \Bigr)
  = \ker \pi .
\]
Therefore $(\pi \circ i_\ell)(x) = 0$.\end{proof}

\subsection{Visibility for twisted elliptic curves at bad reduction}

\begin{definition}\label{J}
Let $E$ and $F$ be elliptic curves, and let
\[
\theta : E[\ell] \xrightarrow{\sim} F[\ell]
\]
be an isomorphism of $\mathrm{Gal}(\overline{\Bbb{Q}}/\Bbb{Q})$-modules. Set
\[
\Delta \coloneqq \{ (a,-\theta(a)) \mid a \in E[\ell] \},
\]
and define the gluing 
\[
J \coloneqq (E \times F)/\Delta.
\]
\newcommand{\qbar}[1]{\overline{\mkern2mu #1 \mkern2mu}}
Then $J$ is a $2$-dimensional abelian variety. Let
\[
\qbar{(\,\cdot\,)}: E\times F \longrightarrow J
\]
be the quotient map. The maps
\[
i_E:E\to J,\quad P\mapsto \qbar{(P,0)},\qquad
i_F:F\to J,\quad Q\mapsto \qbar{(0,Q)}.
\]
are injective homomorphisms, and we regard $E$ and $F$ as abelian subvarieties of $J$ via $i_E$ and $i_F$. Using this construction, we can verify that $i_E(P) = i_F(\theta(P))$ for all $P \in E[\ell]$. In other words, the subgroups $E[\ell]$ and $F[\ell]$ are identified inside $J$ via $\theta$.

\end{definition}

In this section, we prove a visibility theorem that works under the condition $E(\mathbb{Q}_{\ell})[\ell] = 0$. As indicated in Remark \ref{localtorsion}, this property can always be satisfied by taking a quadratic twist with an appropriate $D$. Of course, it is not generally clear whether such a twist can be chosen to be compatible with other conditions of Theorem \ref{visiblemain}, such as the rank. To verify this in concrete examples, we shall use Theorem \ref{lem:RP}.

Throughout the proof, we use the map $\varphi=h\circ\theta\circ\kappa$ defined in Remark~\ref{commutative}. The variety $J$ appears only in the target $\operatorname{Vis}_J(H^1(\mathbb{Q},E))$; thus, the hypotheses of Theorem~\ref{sha} are not required.

\begin{theorem}[A visibility theorem for twisted elliptic curves,
including primes of bad reduction]
(=Theorem\ref{visiblemain})

Let $E,F$ be elliptic curves over $\mathbb{Q}$.
Fix an odd prime number $\ell$.
Assume that $(E,F)$ satisfies the following conditions:
\begin{enumerate}
  \item
  $E[\ell]=F[\ell]$ as $G_{\mathbb{Q}}$-modules.

  \item
  $\displaystyle
  \gcd\!\left(
    \ell,\;
    \prod_{p<\infty}
    c_{E/\mathbb{Q}_p}\cdot c_{F/\mathbb{Q}_p}
  \right)=1
  $
  and $E(\mathbb{Q}_{\ell})[\ell]=0$.

  \item
  $\displaystyle
  \mathrm{rank}(F/\mathbb{Q})
  -\mathrm{rank}(E/\mathbb{Q})\ge 2
  $
  (resp.\ $\displaystyle
  \mathrm{dim}_{\mathbb{F}_{\ell}}
  \mathrm{Sel}^{\ell}(F/\mathbb{Q})
  -\mathrm{rank}(E/\mathbb{Q})\ge 2
  $).
\end{enumerate}

Then the following statements hold:
\begin{enumerate}
  \item[(a)]
  We have
  \[
    \mathrm{Vis}_{J}\bigl(\Sha(E/\mathbb{Q})\bigr)[\ell]
    \neq 0
    \qquad
    \bigl(\text{resp. }\Sha(E/\mathbb{Q})[\ell]\neq 0\bigr).
  \]
  
  \item[(b)]
  Let $D\in\mathbb{Z}$ be a square-free integer such that
  $(E^D,F^D)$ satisfies $(2)$ and the corresponding
  inequality in $(3)$.
  Then
  \[
    \mathrm{Vis}_{J^D}\bigl(\Sha(E^D/\mathbb{Q})\bigr)[\ell]
    \neq 0
    \qquad
    \bigl(\text{resp. }\Sha(E^D/\mathbb{Q})[\ell]\neq 0\bigr).
  \]

\end{enumerate}

\end{theorem}

\begin{proof}
In the key diagram in subsection 4.1, consider the map $\varphi$. We have
\[
\dim_{\mathbb{F}_\ell} \operatorname{Im}\varphi
\ge
\operatorname{rank} F(\mathbb{Q})
-
\operatorname{rank} E(\mathbb{Q}).
\]

When $\ell\nmid p$, we have $\operatorname{res}_p\circ \ \varphi=0$, since the images of $F(\mathbb{Q}_p)/\ell F(\mathbb{Q}_p)$ in $H^1(\mathbb{Q}_p,F[\ell])$ and of $E(\mathbb{Q}_p)/\ell E(\mathbb{Q}_p)$ in $H^1(\mathbb{Q}_p,E[\ell])$ coincide under the identification $H^1(\mathbb{Q}_p,E[\ell])=H^1(\mathbb{Q}_p,F[\ell])$. For $p=\infty$, we have $\operatorname{res}_\infty\circ\ \varphi=0$, since the images of $F(\mathbb{R})/\ell F(\mathbb{R})$ in $H^1(\mathbb{R},F[\ell])$ and of $E(\mathbb{R})/\ell E(\mathbb{R})$ in $H^1(\mathbb{R},E[\ell])$ are both zero under the identification $H^1(\mathbb{R},E[\ell])=H^1(\mathbb{R},F[\ell])$.
Thus, we have

\[
\operatorname{Im}\varphi
\subset
\Sha(E/\mathbb{Q}, \mathrm{Vis})[\ell]\cap \mathrm{Vis}_J(H^1(\Bbb{Q},E)).
\]

By Proposition~\ref{V}, we have
\[
\#\Sha(E/\mathbb{Q})[\ell]
\ge
\frac{\#\Sha(E/\mathbb{Q},\mathrm{Vis})[\ell]}{\ell}
\ge \ell.
\]

Therefore, by the second isomorphism theorem, we have 
\[
\dim_{\mathbb{F}_\ell}
\left(
\operatorname{Im}\varphi
\cap
\Sha(E/\mathbb{Q})[\ell]
\right)
\ge
\dim_{\mathbb{F}_\ell}\operatorname{Im}\varphi
-
\dim_{\mathbb{F}_\ell}
\frac{\Sha(E/\mathbb{Q},\mathrm{Vis})[\ell]}
{\Sha(E/\mathbb{Q})[\ell]}
\ge 2-1=1.
\]

Moreover, since
\[
\operatorname{Im}\varphi
\subset
\operatorname{Vis}_J H^1(\mathbb{Q},E),
\]
we have
\[
0\neq
\operatorname{Im}\varphi
\cap
\Sha(E/\mathbb{Q})[\ell]
\subset
\operatorname{Vis}_J\Sha(E/\mathbb{Q})[\ell]. 
\]

Hence,
\[
\operatorname{Vis}_J\Sha(E/\mathbb{Q})[\ell]\neq 0.
\]

Moreover, in the key diagram, for every $s\in\mathrm{Sel}^{\ell}(F/\mathbb{Q})$,
the class $\Psi(s)$ is locally trivial at every place
$p\neq\ell$. Hence
\[
\mathrm{Im}\Psi\subset R(E/\mathbb{Q}).
\]
By Proposition 4.10 and $E(\mathbb{Q}_{\ell})[\ell]=0$,
we have
\[
\dim_{\mathbb{F}_{\ell}}
\frac{R(E/\mathbb{Q})}{\Sha(E/\mathbb{Q})[\ell]}
\leq 1.
\]
Therefore,
\[
\begin{aligned}
\dim_{\mathbb{F}_{\ell}}\Sha(E/\mathbb{Q})[\ell]
&\geq
\dim_{\mathbb{F}_{\ell}}\mathrm{Im}\Psi
-
\dim_{\mathbb{F}_{\ell}}
\frac{R(E/\mathbb{Q})}{\Sha(E/\mathbb{Q})[\ell]}\\
&\geq 2-1=1.
\end{aligned}
\]

(2) It is enough to prove that $(E^D,F^D)$ again satisfies {\rm (1)} in $J^D$.
As \(G_{\mathbb{Q}}\)-modules, we have
\[
E^{D}[\ell]\cong E[\ell]\otimes \chi_{D},\qquad 
F^{D}[\ell]\cong F[\ell]\otimes \chi_{D},
\]
where \(\chi_{D}\) is the quadratic character attached to \(\mathbb{Q}(\sqrt{D})/\mathbb{Q}\) (see Lemma 2.2 of \cite{CF}).

Since $E[\ell]=F[\ell]$, we have $E^D[\ell]\cong F^D[\ell]$. Thus \((E^D,F^D)\) satisfies condition {\rm (1)}. By assumption, it also
satisfies condition {\rm (2)}. Thus, we obtain $\mathrm{Vis}_{J^D}(\Sha(E^D/\mathbb Q))[\ell]\neq 0$.

\end{proof}

To check whether \(E(\mathbb{Q}_\ell)[\ell]=0\), we use the following theorem.

\begin{theorem}[\cite{RP}]\label{lem:RP}
Let \(E/\mathbb{Q}_p\) be an elliptic curve given by a minimal Weierstrass equation
\[
  y^2 + a_1xy + a_3y = x^3 + a_2x^2 + a_4x + a_6,
\]
and suppose that \(E/\mathbb{Q}_p\) has additive reduction at \(p\) and \(a_i \in p\mathbb{Z}_p\) for \(i\in\{1,2,3,4,6\}\).
Then:
\begin{enumerate}
  \item If \(p=3\),
  \[
    E_0(\mathbb{Q}_3)[3]\neq 0 \quad\Longleftrightarrow\quad a_2 \equiv 6 \bmod{9}.
  \]

 \item  If \(p=5\),
  \[
    E_0(\mathbb{Q}_5)[5]\neq 0 \quad\Longleftrightarrow\quad a_4 \equiv 10 \bmod{25}.
  \]

\item  If \(p=7\),
  \[
    E_0(\mathbb{Q}_7)[7]\neq 0 \quad\Longleftrightarrow\quad a_6 \equiv 14 \bmod{49}.
  \]

\end{enumerate}
\end{theorem}
\begin{proof}
We use the fact that \[
E_0(\mathbb{Q}_p)[p]\neq 0
\quad\Longleftrightarrow\quad
v_p([p](z))\ge 2,\ \forall z\in \widehat{E}(\mathbb{Z}_p)\setminus \widehat{E}(p\mathbb{Z}_p).
\]
Here $[p](z)$ denotes the multiplication-by-$p$ map on the formal group $\widehat{E}$. For the computation of the formal group and for a proof of this theorem, see \cite{RP}.\end{proof}

\begin{remark}\label{localtorsion}
Let $\ell$ be an odd prime number. Let $E/\mathbb{Q}_{\ell}$ be an elliptic curve defined over $\mathbb{Q}_{\ell}$. There exist infinitely many square-free integers $D$ such that $E^D(\mathbb{Q}_{\ell})[\ell]=0$. Indeed, the group $E^D(\mathbb{Q}_{\ell})[\ell] \cong (E[\ell]\otimes \chi_D)^{G_{\mathbb{Q}_{\ell}}}$ is non-trivial if and only if
$\mathrm{Hom}_{G_{\mathbb{Q}_{\ell}}}(\chi_D^{-1},E[\ell])\neq 0$; since $\chi_D$ is quadratic we have $\chi_D^{-1}=\chi_D$, hence this holds if and only if
$E[\ell]$ contains a one-dimensional $G_{\mathbb{Q}_{\ell}}$-subrepresentation isomorphic to $\chi_D$.
As $\dim_{\mathbb
{F}_\ell}E[\ell]=2$, there are at most two quadratic characters $\chi$ such that
$(E[\ell]\otimes \chi)^{G_{\mathbb{Q}_{\ell}}}\neq 0$. On the other hand, for odd $\ell$ we have $\#\bigl(\mathbb{Q}_\ell^\times/(\mathbb{Q}_\ell^\times)^2\bigr)=4$, hence there exists a quadratic character
$\chi_0$ such that
$(E[\ell]\otimes \chi_0)^{G_{\mathbb{Q}_{\ell}}}=0$. Since each square class in $\mathbb{Q}_\ell^{\times}/(\mathbb{Q}_\ell^{\times})^{2}$ contains infinitely many square-free integers, choosing one of the good classes yields infinitely many square-free integers $D$ such that $E^D(\mathbb{Q}_\ell)[\ell]=0$.

\end{remark}

We next give a key example, in the additive reduction case, where the image of the map $\phi$ provided by visibility is in fact not locally trivial at 3. Hence this case cannot be treated by the existing notion of visibility, and it is essential to relax the local condition at 3.

\begin{theorem}(=Theorem \ref{trueexamplejustify})
Assume that the Tate--Shafarevich groups under consideration are finite.
Let $E_1 := 38025.\mathrm{ck}1$ and $F := 4225.\mathrm{k}1$
be as in Example~\ref{example3}.
A Weierstrass model for $F^{1190}$ is given by
\[
F^{1190}\colon y^2 = x^3 - x^2 - 315412673x - 2155874804383.
\]
Then, for $D=1190$, we have
\[
\mathrm{Sel}^3(E_1^D/\mathbb{Q})
\neq \mathrm{Sel}^3(F^D/\mathbb{Q})
\qquad\text{and}\qquad
\Sha(E_1^D/\mathbb{Q})[3] \neq 0.
\]
\end{theorem}

\begin{proof}
For $D = 1190$, the product of all the Tamagawa numbers of $E_1^D$ and of $F^D$ is prime to $3$, and $\operatorname{rank}(F_t^D/\mathbb{Q}) - \operatorname{rank}(E_1^D/\mathbb{Q}) = 3$. Therefore, the map
\[
\phi_{3,D}: F^D(\Bbb{Q}_3)/3F^D(\Bbb{Q}_3)
\to H^1(\Bbb{Q}_3,E_1^D)[3]
\]
is nonzero. Indeed, if $\phi_{3,D}$ is zero, then by the argument of the proof of Theorem \ref{parityeven} the difference of the ranks would be forced to be even.

Let $P:=(26101,2718912)\in F^{1190}(\mathbb{Q})$. A SageMath computation shows that there is no $Q\in F^{1190}(\mathbb{Q}_3)$ such that $P=3Q$. Hence the image of $P$ in
$F^{1190}(\mathbb{Q}_3)/3F^{1190}(\mathbb{Q}_3)$
is nonzero.

Since $F^{1190}(\mathbb{Q}_3)[3]=0$, Mattuck's theorem gives
\[
\#\left(
F^{1190}(\mathbb{Q}_3)/3F^{1190}(\mathbb{Q}_3)
\right)=3.
\]
Since $\phi_{3,1190}$ is nonzero, it is therefore injective. Thus
\[
\phi_{3,1190}\circ\operatorname{res}_3:
F^{1190}(\mathbb{Q})/3F^{1190}(\mathbb{Q})
\longrightarrow
H^1(\mathbb{Q}_3,E_1^{1190})[3]
\]
is nonzero.

In particular, we have
\[
\sha(E_1^{1190}/\Bbb{Q}, \text{Vis})[3]
\neq
\Sha(E_1^{1190}/\Bbb{Q})[3].
\]

By Theorem \ref{visiblemain}, we have $\Sha(E^{1190}_1/\Bbb{Q})[3]\neq 0$. 
\end{proof}

\section{Non-isomorphic pairs of elliptic curves with identical BSD invariants, Kodaira symbols, minimal discriminants, and nontrivial Tate–Shafarevich groups}

In this section, we verify that, for explicit values of $D$ such as $17$, $6977$, and $23297$, the curves $E_1^D$ and $E_2^D$ form a pair of BSD twins whose Tate–Shafarevich groups have isomorphic nontrivial $3$-primary components.

\begin{example}\label{rank}
The quadratic twists of $E_1\coloneqq 38025.ck1$ by $D = 6977$ and $D = 23297$, given by
\[
\begin{aligned}
E^{6977}_1: y^2 &= x^3 - 10306990931527875 x + 402739148904281876618750,\\
E^{23297}_1: y^2 &= x^3 - 114919690409047875 x + 14994005922657755498618750,
\end{aligned}
\]
have rank $0$.

The quadratic twists of the elliptic curve $F\coloneqq 38025.i1$ by $D = 6977$ and $D = 23297$ are
\[
\begin{aligned}
F^{6977} &: y^2 + y = x^3 - 46521826772655x - 122161581370183348094,\\
F^{23297} &: y^2 = x^3 - 8299258575844080x - 291077375250029226118000,
\end{aligned}
\]
and both have rank~$2$.

This is verified using PARI/GP. The output shows $\mathrm{rank}(F^{6977}/\Bbb{Q})=2$. 
For example, one rational point of infinite order on $F^{6977}$ is
\[
\left(x,y\right)
=
\left(
\frac{7600015680280}{609961},
\frac{16724543722010247982}{476379541}
\right).
\]

Thus, for $D = 6977$ and $23297$,
\[
  \operatorname{rank}(F^D/\mathbb{Q}) - \operatorname{rank}(E_1^D/\mathbb{Q}) = 2.
\]
For $t = \frac{1728}{16972865}$ and $D = 17$, we can also verify using PARI/GP that
\[
  \operatorname{rank}(F_t^D/\mathbb{Q}) - \operatorname{rank}(E_1^D/\mathbb{Q}) = 2.
\]
\end{example}

\begin{proposition}\label{tamagawa3}
Let $E/\Bbb{Q}$ be an elliptic curve. Except in the cases where the Kodaira symbol at $p$ is $I_{3n}(n\ge 1)$, $IV$, or $IV^*$, $3$ does not divide the Tamagawa number $c_{E/\Bbb{Q}_p}$.

\end{proposition}

\begin{proof}
Since $3 \nmid \Phi_E(\bar{\Bbb{F}}_p)$ by IV.9, Table 4.1 of 
\cite{silad}, we have $3 \nmid \Phi_E(\Bbb{F}_p)=c_{E/\Bbb{Q}_p}$. 
\end{proof}

\begin{lemma}\label{twistlocal}
Let $E_1$ be the elliptic curve over $\mathbb{Q}$ with LMFDB label $38025\text{.ck1}$.
For a squarefree integer $D$ coprime to $3$, we have
\[
E_1^D(\mathbb{Q}_3)[3]=0
\iff
D\equiv 2 \bmod{3}.
\]
\end{lemma}

\begin{proof}

Since we work with a minimal Weierstrass model and the Tamagawa number is coprime to $3$, $E_1^D(\mathbb{Q}_3)[3]=0 \iff (E_1^D)_0(\mathbb{Q}_3)[3]=0$.
A minimal Weierstrass equation for $38025\text{.ck1}$ is
\[
y^2+xy=x^3-x^2-13233492x+18531699291.
\]
Substituting $y=y'+x$, we obtain
\[
y'^2+3xy'=x^3-3x^2-13233492x+18531699291.
\]
After twisting by a square-free integer $D$, the coefficient of $x^2$ becomes $\dfrac{-3D-9}{4}$.
Note that the quadratic twist by $D$ is still minimal at $3$ because $D$ is coprime to $3$.
Thus, $E^D(\mathbb{Q}_3)$ has nontrivial $3$-torsion if and only if $\dfrac{-3D-9}{4}\equiv 6 \bmod{9}$, by Theorem~\ref{lem:RP} (1).
Equivalently, $D\equiv 1 \bmod{3}$.

\end{proof}

\begin{theorem}\label{nontrivialD}
Let \(E_1\) and \(E_2\) be the elliptic curves with LMFDB labels
\(38025.ck1\) and \(38025.ck2\), respectively.

Let \(F\) be the elliptic curve \(38025.i1\).

Suppose that a square-free integer $D$ satisfies the following conditions:
\begin{enumerate}
\item $D>0$ and $D\equiv 1 \bmod{8}$.
\item $\gcd(D,2\cdot 3\cdot 5\cdot 13)=1$.
\item For every prime $p\mid D$, $\left(\frac{65}{p}\right)=-1$.
\item $D\equiv 2 \bmod 3$, $\mathrm{dim_{\Bbb{F}_3}}\mathrm{Sel^3}(F^D/\Bbb{Q})
-\mathrm{rank}(E_1^D/\Bbb{Q})\ge 2$.
\end{enumerate}
Then $(E_1^D,E_2^D)$ is a BSD twin and
$\Sha(E_1^D/\Bbb{Q})[3]\neq 0$.

If we take $F=F_{\frac{1728}{16972865}}$ as the auxiliary curve
and assume $\left(\frac{D}{257}\right)=1$ in addition to conditions
(1)--(4), then we obtain $\Sha(E_1^D/\Bbb{Q})[3]\neq 0$.

In particular, for $D$ in the following table, we have \(\Sha(E_1^D/\mathbb{Q})[3]\neq 0\). Moreover, the pair \((E_1^D, E_2^D)\) satisfies $\text{BSD}(E_1^D/\Bbb{Q})=\text{BSD}(E_2^D/\Bbb{Q})$ and has the same Kodaira symbols at all primes and the same minimal discriminant, yet \(j(E_1)\neq j(E_2)\).

\FloatBarrier
\begin{center}
\begin{tabular}{c|c}
$D$ & one possible auxiliary curve $F_t$ \\
\hline
$17$ & $F_{\frac{1728}{16972865}}$ \\
$6977,\,23297$ & $38025.\mathrm{i}1$
\end{tabular}
\end{center}

\FloatBarrier

\end{theorem}

\begin{proof}
By Theorem \ref{visiblemain} and Lemma \ref{twistlocal}, it is sufficient to prove that \[\text{gcd}(3, \prod_p {c_{E_1^D/\Bbb{Q}_p} \cdot c_{F^D/\Bbb{Q}_p}})=1.\]

$E_1$, $38025.i1, and F_{\frac{1728}{16972865}}$ have reduction types $I_0^*(p=3), III^*(p=5), III^*(p=13)$, $I_3^*(p=3), III(p=5), III^*(p=13)$, and $I_3^*(p=3),\ III(p=5),\ III(p=13),\ I_6\text{ (non-split)}(p=257)$ respectively. Under quadratic twisting, the Kodaira types $I_0^*$ and $III^*$ cannot become
$IV$, $IV^*$, or $I_{3n}$ ($n \ge 1$). Also, $I_3^*$ cannot become $I_3$ by twisting by $D$ with $\text{gcd}(3,D)=1$ and $I_6$ does not split by $D$ with $\left(\frac{D}{257}\right)=1$. For a prime $p$ such that $\text{gcd}(p, 3\cdot 5 \cdot 13)=1$ and $p\mid D$,  the Kodaira type of $E^D/\Bbb{Q}$ at $p$ is $I_0^*$. By Proposition~\ref{tamagawa3}, together with
$c_{F^D/\Bbb{Q}_{257}}=2$ in the case
$F=F_{\frac{1728}{16972865}}$, we obtain
\[
\gcd\left(3,\prod_p
c_{E_1^D/\Bbb{Q}_p}\,c_{F^D/\Bbb{Q}_p}\right)=1.
\] 
We prove that $\text{BSD}(E^D_1/\Bbb{Q})=\text{BSD}(E^D_2/\Bbb{Q})$ and $\Sha(E^D_1/\Bbb{Q})[3]\cong \Sha(E^D_2/\Bbb{Q}) [3]\neq 0$. 

Let $D$ be a square-free integer satisfying conditions (1)--(4). Then, by \cite{shiga}, \(\mathrm{BSD}(E^D_1/\mathbb{Q}) = \mathrm{BSD}(E^D_2/\mathbb{Q})\) and $\mathrm{rank}(E^D_1/\Bbb{Q})=\mathrm{rank}(E^D_2/\Bbb{Q})=0$. 
By condition (4), we have
\[
\dim_{\Bbb{F}_3}\mathrm{Sel}^3(F^D/\Bbb{Q})
-\operatorname{rank}(E_1^D/\Bbb{Q})\ge 2.
\]

Therefore, by Theorem~\ref{visiblemain}, we have
\[
\Sha(E_1^D/\mathbb{Q})[3] \cong \Sha(E_2^D/\mathbb{Q})[3] \neq 0.
\]\end{proof}

\FloatBarrier

\begin{example}\label{last}

Let \(D>0\) be a square-free integer such that \(D\equiv 1 \bmod{8}\),
\(\gcd(D,2\cdot 3\cdot 5\cdot 13)=1\), and either \(D=1\) or
\(\left(\frac{65}{p}\right)=-1\) for every prime \(p\mid D\).
Assume further that \(D\) is sufficiently large so that
\(E_1^D(\mathbb{Q})_{\mathrm{tors}}\) and
\(E_2^D(\mathbb{Q})_{\mathrm{tors}}\) have no elements of order greater than \(2\). Then the following pair of elliptic curves have the same BSD data, Kodaira symbols, and minimal discriminants. Moreover, if \(D \equiv 2 \bmod{3}\) and
\[
\operatorname{rank}(F^{D}/\mathbb{Q})-\operatorname{rank}(E_1^{D}/\mathbb{Q}) \ge 2,
\] for $F:=38025.i1$, 
then we have \(\Sha(E_1^{D}/\mathbb{Q})[3]\neq 0\) (for example $6977, 23297$). 

\FloatBarrier


\begin{table}[htbp]

\caption{
Non-isomorphic pairs of elliptic curves with identical BSD invariants, Kodaira symbols, minimal discriminants, and nontrivial Tate–Shafarevich groups in Theorem \ref{nontrivialD}.}

\centering
\label{tab:elliptic_curves_38025_b}
\begin{tabular}{lll}
    \toprule
    \textbf{Elliptic curve} & $E^D_1: y^2=x^3+25350Dx^2+2471625D^2x$ & $E^D_2: y^2=x^3-50700Dx^2+632736000D^2x$ \\
    \midrule
 $j$-invariant& $257^3$  & $17^3$  \\  
Minimal discriminant & $3^6\cdot 5^9\cdot 13^9 \cdot D^6$ & $3^6\cdot 5^9\cdot 13^9 \cdot D^6$ \\

    Mordell--Weil group & 
    $\mathbb{Z}/2\mathbb{Z}$ & $\mathbb{Z}/2\mathbb{Z}$ \\
    Regulator & 1 & 1 \\
    Real period & $\dfrac{0.209\cdots}{\sqrt{D}}$ & $\dfrac{0.209\cdots}{\sqrt{D}}$ \\
    Tamagawa number & $2(p=3),2(p=5), 2(p=13), 2(p\mid D)$ & $2(p=3), 2(p=5), 2(p=13), 2(p \mid D)$\\

   Kodaira symbol  & $I_0^*(p=3), III^*(p=5), III^*(p=13)$,  & $I_0^*(p=3), III^*(p=5), III^*(p=13)$ \\

    &  $I_0^* (p \mid D)$ &  $I_0^* (p \mid D)$\\        

    \rowcolor{gray!20}
    $\Sha(E^D_{i}/\mathbb{Q})[2^{\infty}]$ & 0 & 0 \\
    \rowcolor{gray!20}
    $\Sha(E^D_i/\mathbb{Q})$ & \text{Nontrivial 3-torsion} & \text{Nontrivial 3-torsion} \\

    \bottomrule
\end{tabular}

\end{table}

\end{example}

\FloatBarrier

\section{Infinitely many BSD twins that share the same minimal discriminants and Kodaira symbols}

We recall the Gouv\^ea--Mazur construction of twists of rank more than $1$(see Section 2 of \cite{GMazur}) to prove Theorem \ref{mainBSDtwins}.
Let $E:y^2=x^3+Ax+B$ be an elliptic curve.
Let $D(U,V):=U^3V+AUV^3+BV^4$.
Take $(u,v)\in\Bbb{Z}^2$ such that $D(u,v)$ is square-free,
and set $P_{u,v}:=(u/v,1/v^2)$.
Then $P_{u,v}\in E^{D(u,v)}(\Bbb{Q})$.
There are only finitely many such pairs $(u,v)$ for which
$P_{u,v}$ is a torsion point.
Thus, for all but finitely many such pairs $(u,v)$, we have
$\mathrm{rank}(E^{D(u,v)}/\Bbb{Q})\ge 1$.

We recall Bunyakovsky's conjecture
(see \cite{Bunyakovsky}).

\begin{conjecture}[Bunyakovsky]
Let $f(x)\in\mathbb{Z}[x]$ be a nonconstant
polynomial that is irreducible over $\mathbb{Q}$.
Suppose that $\gcd\{f(n)\mid n\in\mathbb{Z}\}=1$. Then $|f(n)|$ is prime for infinitely many
positive integers $n$.
\end{conjecture}

\begin{theorem}\label{mainBSDtwins}
Assume the Bunyakovsky conjecture for $D_k:=3796416000k^3+6840849600k^2-19745252280k-36525288487$.
Then there exist infinitely many BSD twins that have the same Kodaira symbols and minimal discriminants $3$-primary part of the Tate--Shafarevich group containing $(\Bbb{Z}/3\Bbb{Z})^2$. \end{theorem}

\begin{proof}
For $F:=38025.i1$, given by
$y^2=x^3-15291120x-23020166000$,
the Gouv\^ea--Mazur theorem provides twists $F^D$ of rank
at least $1$.

For $D(u,v)$ in the Gouv\^ea--Mazur construction, take $v=1$ and
\[
u\equiv 1\bmod{8},\qquad
u\equiv 1\bmod{3},\qquad
u\equiv 2\bmod{5},\qquad
u\equiv 1\bmod{13}.
\]
We may therefore take $u:=1560k+937$ with $k\ge 0$.
For this $u$, set $D_k:=D(u,1)=(1560k+937)^3-15291120(1560k+937)-23020166000=3796416000k^3+6840849600k^2-19745252280k-36525288487$.

Since $6k^3+4k^2+2k+5$ has no roots in $\mathbb{F}_7$ and
$f(k)\equiv 6k^3+4k^2+2k+5\bmod{7}$, $f$ is irreducible over $\mathbb{Q}$. Moreover, $\gcd(f(0),f(1))=\gcd(-36525288487,-45633275167)=1$,
so no prime divides $f(k)$ for every integer $k$. Since $f$ also has a positive leading coefficient, it satisfies the hypotheses of Bunyakovsky's conjecture.

Choose a sufficiently large $k$ such that $D_k$ is prime.

Then $D_k$ satisfies conditions (1)--(3) of
Theorem \ref{nontrivialD}. Thus, $\mathrm{rank}(E^D_1/\Bbb{Q})=0$ by Proposition 6.3 of \cite{shiga}.

Now, by the Dokchitsers' parity theorem, we have
\[
(-1)^{\dim_{\mathbb{F}_3}\mathrm{Sel}^3(F^{D_k}/\mathbb{Q})
-\dim_{\mathbb{F}_3}F^{D_k}(\mathbb{Q})[3]}
=
(-1)^{\operatorname{corank}_{\mathbb{Z}_3}
\mathrm{Sel}^{3^\infty}(F^{D_k}/\mathbb{Q})}
=
\omega(F^{D_k})
=
\omega(E_1^{D_k})
=1.
\]
Since $F^{D_k}(\mathbb{Q})[3]=E_1^{D_k}(\mathbb{Q})[3]=0$,
we see that
$\dim_{\mathbb{F}_3}\mathrm{Sel}^3(F^{D_k}/\mathbb{Q})$
is even.
By the Gouv\^ea--Mazur construction, we have
$\operatorname{rank}(F^{D_k}/\mathbb{Q})\ge 1$.
Therefore,
\[
\dim_{\mathbb{F}_3}\mathrm{Sel}^3(F^{D_k}/\mathbb{Q})
-\operatorname{rank}(E_1^{D_k}/\mathbb{Q})\ge 2,
\]
and hence
$\Sha(E_1^{D_k}/\mathbb{Q})[3]$ contains $(\Bbb{Z}/3\Bbb{Z})^2$. 

Since $\omega(E_1^{D_k})=1$, we have
\[
(-1)^{\dim_{\mathbb{F}_3}\mathrm{Sel}^3(E_1^{D_k}/\mathbb{Q})
-\dim_{\mathbb{F}_3}E_1^{D_k}(\mathbb{Q})[3]}
=
(-1)^{\operatorname{corank}_{\mathbb{Z}_3}
\mathrm{Sel}^{3^\infty}(E_1^{D_k}/\mathbb{Q})}
=
\omega(E_1^{D_k}).
\]
Thus, since $\operatorname{rank}(E_1^{D_k}/\mathbb{Q})=0$
and $E_1^{D_k}(\mathbb{Q})[3]=0$, we have
\[
\Sha(E_1^{D_k}/\mathbb{Q})[3]
\cong \mathrm{Sel}^3(E_1^{D_k}/\mathbb{Q}).
\]
Hence, $\Sha(E_1^{D_k}/\mathbb{Q})$ contains a subgroup
isomorphic to $(\mathbb{Z}/3\mathbb{Z})^2$.\end{proof}

\begin{example}
A computation using Magma shows that \(D_k\) is prime for exactly \(492\) integers \(k\) in the range \(0\leq k\leq3576\). The first ten such values of \(k\) are $5,\ 6,\ 7,\ 16,\ 22,\ 28,\ 32,\ 33,\ 41,\ 45$. For example, \(D_5=510321690113\) is prime.
\end{example}

\subsection{Acknowledgements}
I would like to express my heartfelt gratitude to Prof.\ Nobuo Tsuzuki for his many constructive discussions and continual encouragement throughout this work. I also thank Prof.\ Christian Wuthrich for insightful discussions at Y-RANT in Nottingham about my previous paper \cite{shiga} and the nontriviality of the Tate--Shafarevich group. I am also grateful to Prof. Go Yamashita for a question during the Number Theory Seminar at Kyoto University, which motivated Remark \ref{commutative}. I would like to thank Naoto Dainobu for his helpful discussions.


\begin{thebibliography}{99}




\bibitem{2013} A. Agashe and S. Biswas, Constructing non-trivial elements of the Shafarevich--Tate group of an abelian variety over a number field, J. Number Theory 133 (2013), no. 6, 1977--1990.


\bibitem{visibility} A. Agashe and W. Stein, Visibility of Shafarevich–Tate groups of abelian varieties, J. Number Theory 97 (2002), no. 1, 171–185.

\bibitem{mazur} A. Agashe and W. Stein, Visible evidence for the Birch and Swinnerton–Dyer conjecture for modular abelian varieties of analytic rank zero, Math. Comp. 74 (2005), no. 249, 455–484.



\bibitem{twin}
A. J. Barrios, M. Brucal-Hallare, A. Deines, P. Harris, and M. Roy,
Prime isogenous discriminant ideal twins,
Journal of Number Theory, in press, 2026.



\bibitem{aj} A.J Barrios, M. Roy, N. Sahajpal, D. Tallana, B. Tobin, H. Wiersema. Local data of elliptic curves under quadratic twist. Res. number theory 11, 75 (2025). https://doi.org/10.1007/s40993-025-00650-w




\bibitem{Bhargava} M. Bhargava, Z. Klagsbrun, R. J. Lemke Oliver and A. Shnidman, Elements of given order in Tate–Shafarevich groups of abelian varieties in quadratic twist families, Algebra Number Theory 15 (2021), no. 3, 627–655.



\bibitem{Ca} J.W.S. Cassels, Arithmetic on curves of genus 1. VII. The dual exact sequence. Journal für die reine und angewandte Mathematik 216: 150-158, 1964. 


\bibitem{CF} J. E. Cremona and N. Freitas, Global methods for the symplectic type of congruences between elliptic curves, Rev. Mat. Iberoam. 38 (2022), no. 1, 1--32, doi:10.4171/RMI/1269.





\bibitem{Ma} J. E. Cremona and B. Mazur, Visualizing elements in the Shafarevich–Tate group, Experiment. Math. 9 (2000), no. 1, 13–28.



\bibitem{parity} T.~Dokchitser and V.~Dokchitser,
\newblock On the Birch-Swinnerton-Dyer quotients modulo squares,
\newblock { Annals of Mathematics} {\bf 172} (2010), no.~1, 567--596.




\bibitem{Fisher}
T.~Fisher,
Visualizing elements of order $7$ in the Tate--Shafarevich group of an elliptic curve, LMS J. Comput. Math. 19 (Special Issue~A) (2016), 100--114.


\bibitem{GMazur}
F. Gouv\^ea and B. Mazur,
The square-free sieve and the rank of elliptic curves,
J. Amer. Math. Soc. \textbf{4} (1991), no.~1, 1--23.





\bibitem{flat}
K{\k e}stutis {\v C}esnavi{\v c}ius, Selmer groups as flat cohomology groups, J. Ramanujan Math. Soc.{31} (2016), no.~1, 31--61.






\bibitem{lmfdb} The LMFDB Collaboration. (2025). The L-functions and modular forms database. \url{https://www.lmfdb.org}. 




\bibitem{Mattuck}  A.~Mattuck, ``Abelian varieties over $p$-adic ground fields'', Ann.\ of Math. (2) \textbf{62} (1955), 92--119. MR \; Zbl.


\bibitem{Bunyakovsky} P. Ribenboim, The Little Book of Big Primes, Springer-Verlag, New York, 1991.

\bibitem{RP} R. Pannekoek, On $p$-torsion of $p$-adic elliptic curves with additive reduction, arXiv:1211.5833 (2013).

\bibitem{RS95}
K.~Rubin and A.~Silverberg,
Families of elliptic curves with constant mod $p$ representations,
in Elliptic Curves, Modular Forms, \& Fermat's Last Theorem
(Hong Kong, 1993),
Series in Number Theory, vol.~I,
International Press, Cambridge, MA, 1995, pp.~148--161.



\bibitem{descent} Edward F. Schaefer and Michael Stoll, “How to Do a p-Descent on an Elliptic Curve,” Transactions of the American Mathematical Society 356 (2004), no. 3, 1209–1231.


\bibitem{shiga} A. Shiga, Infinitely many pairs of non-isomorphic elliptic curves sharing the same BSD invariants, arXiv:2507.18574 (2025).


\bibitem{behavior} A. Shiga, Behaviors of the Tate--Shafarevich group of elliptic curves under quadratic field extensions, Tokyo J. Math. \textbf{49} (2026), no.~1, 199--224.




\bibitem{Ariel} A. Shnidman and A. Weiss, Elements of prime order in Tate–Shafarevich groups of abelian varieties over $\mathbb{Q}$, Forum Math. Sigma 10 (2022), e98, 1–10.




\bibitem{silad} J. H. Silverman, Advanced topics in the arithmetic of elliptic curves, Graduate Texts in Mathematics, vol. 151, Springer-Verlag, New York, 1994.


\bibitem{stein} W. Stein, Modular Forms, a Computational Approach (Graduate Studies in Mathematics, 79).

\bibitem{SW13}
W.~Stein and C.~Wuthrich,
\newblock Algorithms for the arithmetic of elliptic curves using Iwasawa theory,
\newblock Math.\ Comp. 82 (2013), no.~283, 1757--1792,
\newblock doi:\,10.1090/S0025-5718-2012-02649-4.



\bibitem{Wiese2013}
G.~Wiese, Modular Galois Representations and Applications,
\newblock lecture notes (4 lectures held at the Higher School of Economics in Moscow, 2--4 April 2013), Version of 6 April 2013.




\end{thebibliography}
\end{document}